\documentclass{article} 
\usepackage{proceed2e}

\usepackage{amssymb,amsfonts}
\usepackage{amsthm} 
\usepackage[cmex10]{amsmath}
\usepackage[usenames,dvipsnames,svgnames]{xcolor} 
\usepackage{graphicx}
\usepackage{url}
\usepackage{times}
\usepackage{array} 
\usepackage[sort]{natbib} 

\usepackage{algorithm,algpseudocode} 

\usepackage[utf8]{inputenc}
\usepackage[T1]{fontenc}
\usepackage{lmodern} 

\usepackage{hyperref}
\hypersetup{pdfauthor={Stephen Becker, Volkan Cevher, 
    Sasha Aravkin, Peder Olsen},
    pdftitle={A variational approach to stable principal component pursuit},
    pdfkeywords={rpca, robust PCA, SPCP, stable principal component pursuit, source separation, quasi-Newton},
    colorlinks=true,
    citecolor=Mulberry,
    urlcolor=Bittersweet,
}
\newtheorem{theorem}{Theorem}[section]
\newtheorem{lemma}[theorem]{Lemma}
\newtheorem{corollary}[theorem]{Corollary}
\newtheorem{proposition}[theorem]{Proposition}

\newtheorem{remark}[subsection]{Remark}

\numberwithin{equation}{section}

\let\oldmarginpar\marginpar
\renewcommand\marginpar[1]{\-\oldmarginpar[\raggedleft\footnotesize #1]%
{\raggedright\footnotesize #1}}

\usepackage[normalem]{ulem} 

\DeclareMathOperator{\diag}{diag}

\DeclareMathOperator*{\minimize}{minimize}

\newcommand{\eps}{\varepsilon}

\newcommand{\R}{\mathbb{R}}

\newcommand{\AAA}{\ensuremath{\mathcal{A}}}   
\newcommand{\proj}{\ensuremath{\mathcal{P}}}  
\newcommand{\Proj}{\proj}

\newcommand{\order}{\mathcal{O}}              

\DeclareMathAlphabet{\mathbbb}{U}{bbold}{m}{n}

\newcommand{\iprod}[2]{\left\langle #1,\,#2 \right\rangle}

\newcommand{\norm}[1]{{\left\lVert{#1}\right\rVert}}

\newcommand{\gauge}[2]{\gamma\left(#1 \mid #2 \right)}

\newcommand{\pmat}[1]{\begin{pmatrix}#1\end{pmatrix}}
\newcommand{\vertiii}[1]{{\left\vert\kern-0.25ex\left\vert\kern-0.25ex\left\vert #1
    \right\vert\kern-0.25ex\right\vert\kern-0.25ex\right\vert}}

\newcommand{\matrixNorm}[1]{\vertiii{#1}} 

\newcommand{\lambdaSum}{\lambda_\text{sum}}
\newcommand{\lambdaMax}{\lambda_\text{max}}
\newcommand{\lambdaL}{\lambda_\text{L}}
\newcommand{\lambdaS}{\lambda_\text{S}}
\newcommand{\tauSum}{\tau_\text{sum}}
\newcommand{\tauMax}{\tau_\text{max}}

\newcommand{\phiSum}{\phi_{\text{sum}}}
\newcommand{\phiMax}{\phi_{\text{max}}}

\newcommand{\Loracle}{L_\text{oracle}}
\newcommand{\Soracle}{S_\text{oracle}}

\title{A  variational approach to stable principal component pursuit
}
\author{ {\bf Aleksandr Aravkin} \\
T. J. Watson Center \\
IBM Research \\
Yorktown Heights, NY \\
\And
{\bf Stephen Becker} \\
T. J. Watson Center \\
IBM Research \\
Yorktown Heights, NY \\
\And
{\bf Volkan Cevher} \thanks{\ Author's work is supported in part by the European Commission under the grants MIRG-268398 and ERC Future Proof, and by the Swiss Science Foundation under the grants SNF 200021-132548, SNF 200021-146750 and SNF CRSII2-147633.} \\
LIONS \\
EPFL \\
Lausanne, Switzerland \\
\And
{\bf Peder Olsen} \\
T. J. Watson Center \\
IBM Research \\
Yorktown Heights, NY \\
}
\date{\today}
\begin{document}
\maketitle

\begin{abstract}
We introduce a new convex formulation for stable principal component
pursuit (SPCP) to decompose noisy signals into low-rank and sparse
representations. For numerical solutions of our SPCP formulation, we
first develop a convex variational framework and then accelerate it
with quasi-Newton methods.  We show, via synthetic and real data
experiments, that our approach offers advantages over the classical
SPCP formulations in scalability and practical parameter selection.
\end{abstract}

\section{INTRODUCTION}
Linear superposition is  a useful model for many applications, including
nonlinear mixing problems. Surprisingly, we can perfectly distinguish multiple
elements in a given signal using convex optimization as long as they are
concise and look sufficiently different from one another. Popular examples
include robust principal component analysis (RPCA) where we decompose a signal
into low rank and sparse components and {\it stable principal component pursuit
(SPCP)}, where we also seek an explicit noise component within the RPCA
decomposition. Applications include alignment of occluded
images~\citep{PenGanWri:12}, scene triangulation~\citep{ZhaLiaGan:11},  model
selection~\citep{CPW12}, face recognition, and document
indexing~\citep{CanLiMa:11}.

The SPCP formulation  can be mathematically stated as follows.  Given a noisy
matrix $Y \in \R^{m \times n}$, we decompose it as a sum of a low-rank matrix
$L$ and a sparse matrix $S$ via the following convex program
\begin{equation} \label{eq:sum-SPCP} \tag{$\text{SPCP}_\text{sum}$}
\begin{aligned}
   & \minimize_{L,S}\; \matrixNorm{L}_* + \lambdaSum \|S\|_1 \\
   & \text{subject to}\; \|L+S-Y\|_F \le \eps, 
\end{aligned}
\end{equation}
where the 1-norm $\|\cdot\|_1$ and nuclear norm $\matrixNorm{\cdot}_*$ are
given by $\|S\|_1 = \sum_{i,j} |s_{i,j}|, \matrixNorm{L}_* = \sum_{i}
\sigma_i(L),$ where $\sigma(L)$ is the vector of singular values of $L$. In
\eqref{eq:sum-SPCP}, the parameter $\lambdaSum > 0$ controls the relative
importance of the low-rank term $L$ vs.\ the sparse term $S$, and the parameter
$\eps$ accounts for the unknown perturbations $Y-(L+S)$ in the data not
explained by $L$ and $S$. 

When $\eps=0$,~\eqref{eq:sum-SPCP} is the ``robust PCA'' problem as analyzed
by~\cite{ChaSanPar:09,CanLiMa:11}, and it has perfect recovery guarantees under
stylized incoherence assumptions. There is even theoretical guidance for
selecting a minimax optimal regularization parameter $\lambdaSum$
\citep{CanLiMa:11}. Unfortunately, many practical problems only approximately
satisfy the idealized assumptions, and hence, we typically tune RPCA via
cross-validation techniques. SPCP further complicates the practical tuning due
to the additional parameter $\eps$. 

To cope with practical tuning issues  of SPCP, we propose the following new variant called ``max-SPCP'':
\begin{equation} \label{eq:max-SPCP} \tag{$\text{SPCP}_\text{max}$}
\begin{aligned}
   & \minimize_{L,S}\; \max\left( \matrixNorm{L}_* , \lambdaMax \|S\|_1 \right) \\
   & \text{subject to}\; \|L+S-Y\|_F \le \eps,
\end{aligned}
\end{equation}
where $\lambdaMax>0$ acts similar to $\lambdaSum$. Our work shows that this new
formulation  offers both modeling and computational advantages over
\eqref{eq:sum-SPCP}. 

Cross-validation with \eqref{eq:max-SPCP} to estimate $(\lambdaMax,\eps)$ is
significantly easier than estimating $(\lambdaSum,\eps)$ in
\eqref{eq:sum-SPCP}. For example, given an \emph{oracle} that provides an ideal
separation $Y \simeq \Loracle + \Soracle$, we can use $\eps = \|\Loracle  +
\Soracle - Y \|_F$ in both cases. However, while we can estimate $\lambdaMax =
\|\Loracle\|_*/\|\Soracle\|_1$, it is not clear how to choose $\lambdaSum$ from
data. Such cross validation can be performed on a similar dataset, or it could
be obtained from a probabilistic model.

Our convex approach for solving \eqref{eq:sum-SPCP} generalizes to other source
separation problems \citep{McCoySPMAG} beyond SPCP. Both  \eqref{eq:max-SPCP}
and \eqref{eq:sum-SPCP} are challenging to solve when the dimensions are large.
We show in this paper that these problems can be solved more efficiently by
solving a few (typically 5 to 10) subproblems of a different functional form.
While the efficiency of the solution algorithms for \eqref{eq:sum-SPCP} relies
heavily on the efficiency of the 1-norm and nuclear norm projections, the
efficiency of our solution algorithm \eqref{eq:max-SPCP} is preserved for
arbitrary norms. Moreover,  \eqref{eq:max-SPCP} allows a faster algorithm in
the standard case, discussed in Section~\ref{sec:QN}.

\section{A PRIMER ON SPCP}
The theoretical and algorithmic research on SPCP formulations (and source
separation in general)  is rapidly evolving. Hence, it is important to set the
stage first in terms of the available formulations to highlight our
contributions. 

To this end, we  illustrate \eqref{eq:sum-SPCP} and \eqref{eq:max-SPCP} via
different convex formulations. Flipping the objective and the constraints
in~\eqref{eq:max-SPCP} and~\eqref{eq:sum-SPCP}, we obtain the following convex
programs
\begin{equation} \label{eq:sum-SPCP-flip}  \tag{$\text{flip-SPCP}_\text{sum}$}
\begin{aligned}
   & \minimize_{L,S}\; \frac{1}{2}\|L+S-Y\|_F^2  \\
   & \text{s.t.}\quad \matrixNorm{L}_* + \lambdaSum\|S\|_1 \le \tauSum
\end{aligned}
\end{equation}

\begin{equation} \label{eq:max-SPCP-flip}  \tag{$\text{flip-SPCP}_\text{max}$}
\begin{aligned}
   & \minimize_{L,S}\; \frac{1}{2}\|L+S-Y\|_F^2  \\
   & \text{s.t.}\quad \max(\matrixNorm{L}_*, \lambdaMax\|S\|_1) \le \tauMax
\end{aligned}
\end{equation}

\begin{remark}
\label{rem:flip}
The solutions of~\eqref{eq:sum-SPCP-flip} and~\eqref{eq:max-SPCP-flip} are
 related to the solutions of~\eqref{eq:sum-SPCP}
and~\eqref{eq:max-SPCP} via the Pareto frontier by~\citet[Theorem
2.1]{AravkinBurkeFriedlander:2013}.
If the constraint $\|L + S - Y\| \leq \eps$ is tight at the solution, 
 then there exist corresponding parameters $\tauSum(\eps)$ and
$\tauMax(\eps)$, for which the optimal value of~\eqref{eq:sum-SPCP-flip}
and~\eqref{eq:max-SPCP-flip} is $\eps$, and the corresponding optimal solutions
$(\overline S_s, \overline L_s)$ and $(\overline S_m, \overline L_m)$ are also
optimal for~\eqref{eq:sum-SPCP}  and~\eqref{eq:max-SPCP}. 
\end{remark}

For completeness, we also include the Lagrangian formulation, which is covered by our new algorithm:
\begin{equation} \label{eq:lag-SPCP} \tag{Lag-SPCP}
   \minimize_{L,S}\; \lambdaL \matrixNorm{L}_* + \lambdaS \|S\|_1 + \frac{1}{2}\|L+S-Y\|_F^2 
\end{equation}

Problems \eqref{eq:max-SPCP-flip} and \eqref{eq:sum-SPCP-flip} can be solved using 
projected gradient and accelerated gradient methods. 
The disadvantage of some of these formulations is that it may not be clear how to tune the parameters.  
Surprisingly, an algorithm we
propose in this paper can solve~\eqref{eq:max-SPCP} and~\eqref{eq:sum-SPCP}
using a sequence of flipped problems that specifically exploits the structured
relationship cited in Remark~\ref{rem:flip}. In practice, we will see that
better tuning also leads to faster algorithms, e.g., fixing $\eps$ ahead of
time to an estimated `noise floor' greatly reduces the amount of required
computation if parameters are to be selected via cross-validation.

Finally, we note that in some cases, it is useful to change the $\|L+S-Y\|_F$
term to $\|\AAA(L+S-Y)\|_F$ where $\AAA$ is a linear operator. For example, let
$\Omega$ be a subset of the indices of a $m \times n$ matrix. We may only
observe $Y$ restricted to these entries, denoted $\Proj_\Omega(Y)$, in which
case we choose $\AAA=\Proj_\Omega$. Most existing RPCA/SPCP algorithms adapt to
the case $\AAA=\Proj_\Omega$ but this is due to the strong properties of the
projection operator $\Proj_\Omega$. The advantage of our approach is that it
seamlessly handles arbitrary linear operators $\AAA$. In fact, it also
generalizes to smooth misfit penalties, that are more robust than the Frobenius
norm, including the Huber loss. Our results also generalize to some other
penalties on $S$ besides the 1-norm.

The paper proceeds as follows. In Section~\ref{sec:literature}, we describe
previous work and algorithms for SPCP and RPCA.  In
Section~\ref{sec:variational}, we cast the relationships between pairs of
problems~\eqref{eq:sum-SPCP-flip}, \eqref{eq:sum-SPCP} and
\eqref{eq:max-SPCP-flip},~\eqref{eq:max-SPCP} into a general variational
framework, and highlight the product-space regularization structure  that
enables us solve the formulations of interest using corresponding flipped
problems.  We discuss computationally efficient projections as optimization
workhorses in Section~\ref{sec:projections}, and develop new accelerated
projected quasi-Newton methods for the flipped and Lagrangian formulations in
Section~\ref{sec:QN}. Finally, we demonstrate the efficacy of the new solvers
and the overall formulation on synthetic problems and a real cloud removal
example in Section~\ref{sec:numerics}, and follow with conclusions in
Section~\ref{sec:conclusions}.

\section{PRIOR ART}
\label{sec:literature}
While problem \eqref{eq:sum-SPCP} with $\eps=0$ has several solvers (e.g., it
can be solved by applying the widely known Alternating Directions Method of
Multipliers (ADMM)/Douglas-Rachford method~\citep{Combettes2007}), the
formulation assumes the data are noise free. Unfortunately, the presence of
noise we consider in this paper introduces a third term in the ADMM framework,
where the algorithm is shown to be non-convergent \citep{chen2013direct}.
Interestingly,  there are only a handful of methods that can handle this case.
Those using smoothing techniques no longer promote exactly sparse and/or
exactly low-rank solutions~\citep{AybatRPCA}. 
Those using dual decomposition techniques require
high iteration counts.  Because each step requires a partial singular value
decomposition (SVD) of a large matrix, it is critical that the methods only
take a few iterations.

As a rough comparison, we start with related solvers that solve
\eqref{eq:sum-SPCP} for  $\eps=0$.  \cite{RPCA_algo_Wright} solves an instance
of \eqref{eq:sum-SPCP} with $\eps=0$ and a $800 \times 800$ system in $8$
hours. By switching to the \eqref{eq:lag-SPCP} formulation, \cite{GaneshRPCA}
uses the accelerated proximal gradient method~\citep{BecTeb09} to solve a $1000
\times 1000$ matrix in under one hour. This is improved further in
\cite{lin2010augmented} which again solves \eqref{eq:sum-SPCP} with $\eps=0$
using the augmented Lagrangian and ADMM methods and solves a $1500\times 1500$
system in about a minute. As a prelude to our results, our method can solve
some systems of this size in about $10$ seconds (c.f.,~Fig.~\ref{fig:1}).

In the case of \eqref{eq:sum-SPCP} with $\eps > 0$, \cite{ASALM} propose the
alternating splitting augmented Lagrangian method (ASALM), which exploits
separability of the objective in the splitting scheme, and can solve a $1500
\times 1500$ system in about five minutes.

The partially smooth proximal gradient (PSPG) approach of \cite{AybatRPCA}
smooths just the nuclear norm term and then applies the well-known FISTA
algorithm~\citep{BecTeb09}.  \cite{AybatRPCA} show that the proximity step can
be solved efficiently in closed-form, and the dominant cost at every iteration
is that of the partial SVD.  They include some examples on video,  
lopsided matrices: $25000 \times 300$ or so, in about 1 minute).  solving $1500
\times 1500$ formulations in under half a minute.

The nonsmooth adaptive Lagrangian (NSA) algorithm of~\cite{Aybat2013} is a
variant of the ADMM for \eqref{eq:sum-SPCP}, and makes use of the insight
of~\cite{AybatRPCA}.  The ADMM variant is interesting in that it splits the
variable $L$, rather than the sum $L+S$ or residual $L+S-Y$.  Their experiments
solve a 1500 $\times$ 1500 synthetic problems in between 16 and 50 seconds
(depending on accuracy) .

\citet{Shen2014} develop a method exploiting low-rank matrix factorization
scheme, maintaining $L = UV^T$.  This technique has also been effectively used
in practice for matrix completion~\citep{JasonLee,recht2011parallel,Aravkin2013}, but lacks a full
convergence theory in either context.  The method of~\citep{Shen2014} was an
order of magnitude faster than ASALM, but encountered difficulties in some
experiments where the sparse component dominated the low rank component in some
sense.  \cite{Mansour2014} attack the~\ref{eq:sum-SPCP} formulation using a factorized approach, 
together with alternating solves between $(U, V)$ and $S$. Non-convex techniques also include 
hard thresholding approaches, e.g. the approach of~\cite{kyrillidis2014matrix}.
While the factorization technique may potentially speed up some
of the methods presented here, we leave this to future work, and only work
with convex formulations. 

\section{VARIATIONAL FRAMEWORK}
\label{sec:variational}

Both of the formulations of interest 
\eqref{eq:sum-SPCP} and~\eqref{eq:max-SPCP} can be written as follows:
\begin{equation}
\label{eq:general}
\min \phi(L, S) \quad \text{s.t.} \quad \rho\left(L + S - Y\right) \leq \eps. 
\end{equation}
Classic formulations assume $\rho$ to be the Frobenius norm; however, this
restriction is not necessary, and we consider $\rho$ to be smooth and convex.
In particular, $\rho$ can be taken to be the robust Huber penalty~\citep{Hub}.
Even more importantly, this formulation allows pre-composition of a smooth
convex penalty with an arbitrary linear operator $\mathcal A$, which extends
the proposed approach to a much more general class of problems.  Note that a
simple operator is already embedded in both formulations of interest:  
\begin{equation}
\label{eq:ids}
L + S = \begin{bmatrix} I & I \end{bmatrix} \begin{bmatrix} L \\ S \end{bmatrix}.
\end{equation}
Projection onto a set of observed indices $\Omega$ is also a simple linear
operator that can be included in $\rho$.  Operators may include different
transforms (e.g., Fourier) applied to either $L$ or $S$. 

The main formulations of interest differ only in the functional $\phi(L,S)$.  For~\eqref{eq:sum-SPCP}, we have 
\[
\phiSum(L, S) = \matrixNorm{L}_* + \lambdaSum \|S\|_1,
\] 
while for~\eqref{eq:max-SPCP}, 
\[
\phiMax(L, S) = \max(\matrixNorm{L}_*, \lambdaMax \|S\|_1).
\]
The problem class~\eqref{eq:general} falls into the class of problems studied
by~\cite{vandenberg2008probing,BergFriedlander:2011} for $\rho(\cdot) =
\|\cdot\|^2$ and by~\cite{AravkinBurkeFriedlander:2013} for arbitrary convex
$\rho$.  Making use of this framework, we can define a value function
\begin{equation}
\label{eq:value}
v(\tau) = \min_{L, S} \rho\left(\mathcal A (L, S) - Y\right) \quad
\text{s.t. } \phi(L,S) \leq \tau,
\end{equation}
and use Newton's method to find a solution to $v(\tau) = \eps$.  The approach
is agnostic to the linear operator $\mathcal A$ (it can be of the simple
form~\eqref{eq:ids}; include restriction in the missing data case, etc.).

For both formulations of interest, $\phi$ is a norm defined on a product space
$\mathbb{R}^{n\times m} \times \mathbb{R}^{n\times m}$, since we can write 
\begin{eqnarray}
\phiSum(L, S) &= \left\|\begin{matrix} \matrixNorm{L}_* \\ \lambdaSum \|S\|_1 \end{matrix}\right\|_1, \\
\phiMax(L,S) &=  \left\|\begin{matrix} \matrixNorm{L}_* \\ \lambdaMax \|S\|_1 \end{matrix}\right\|_\infty.
\end{eqnarray}
In particular, both $\phiSum(L,S)$ and $\phiMax(L,S)$ are {\it gauges}.  For a
convex set $C$ containing the origin, the gauge $\gauge{x}{C}$ is defined by
\begin{equation}
\label{eq:gauge}
\gauge{x}{C} = \inf_\lambda \{\lambda: x \in \lambda C\}.
\end{equation}
For any norm $\|\cdot\|$, the set defining it as a gauge is simply the unit
ball $\mathbb{B}_{\|\cdot\|} = \{x: \|x\|\leq 1\}$.  We introduce gauges for
two reasons. First, they are more general (a gauge is a norm only if $C$ is
bounded with nonempty interior and symmetric about the origin). For example,
gauges trivially allow inclusion of non-negativity constraints.  Second,
definition~\eqref{eq:gauge} and the explicit set $C$ simplify the exposition of
the following results. 

In order to implement Newton's method for~\eqref{eq:value}, the optimization
problem to evaluate $v(\tau)$ must be solved (fully or approximately) to obtain
$(\overline L, \overline S)$.  Then the $\tau$ parameter for the
next~\eqref{eq:value} problem is updated via 
\begin{equation}
\label{eq:newton}
\tau^{k+1} = \tau^k - \frac{v(\tau) - \tau}{v'(\tau)}. 
\end{equation}
Given $(\overline L, \overline S)$, $v'(\tau)$ can be written in closed form using
\cite[Theorem 5.2]{AravkinBurkeFriedlander:2013}, which simplifies to 
\begin{equation}
\label{eq:vprime}
v'(\tau) = -\phi^\circ(\mathcal{A}^T\nabla \rho(\mathcal A (\overline L, \overline S) - Y)), 
\end{equation}
with $\phi^\circ$ denoting the polar gauge to $\phi$. 
The polar gauge is precisely $\gauge{x}{C^\circ}$, with
\begin{equation}
\label{eq:polar}
C^\circ = \{v: \left \langle v, x\right\rangle \leq 1 \quad \forall x \in C\}.
\end{equation}

In the simplest case, where $\mathcal A$ is given by~\eqref{eq:ids}, and $\rho$
is the least squares penalty, the formula~\eqref{eq:vprime} becomes 
\[
v'(\tau) = -\phi^\circ\left(\begin{bmatrix} \overline L + \overline S - Y \\ \overline L + \overline S - Y  \end{bmatrix}\right). 
\]

The main computational challenge in the approach outlined
in~\eqref{eq:value}-\eqref{eq:vprime} is to design a fast solver to evaluate
$v(\tau)$. Section~\ref{sec:QN} does just this. 

%

The key to RPCA is that the regularization functional $\phi$ is a gauge over
the product space used to decompose $Y$ into summands $L$ and $S$.  This makes
it straightforward to compute polar results for both $\phiSum$ and $\phiMax$.

\begin{theorem}[Max-Sum Duality for Gauges on Product Spaces]
\label{thm:product-gauge}
Let $\gamma_1$ and $\gamma_2$ be gauges on $\mathbb{R}^{n_1}$ and $\mathbb{R}^{n_2}$, and consider the function 
\[
g(x,y) = \max\{\gamma_1(x), \gamma_2(y)\}. 
\]
Then $g$ is a gauge, and its polar is given by 
\[
g^\circ(z_1,z_2) = \gamma_1^\circ(z_1) + \gamma_2^\circ(z_2). 
\]
\end{theorem}
\begin{proof}
Let $C_1$ and $C_2$ denote the canonical sets corresponding to gauges $\gamma_1$ and $\gamma_2$. It immediately follows 
that $g$ is a gauge for the set $C = C_1 \times C_2$, since 
\[
\begin{aligned}
\inf\{\lambda \geq 0| (x,  y) \in \lambda C\} &= \inf\{\lambda |  x \in \lambda C_1 \text{ and }y \in \lambda C_2\} \\
&= \max \{\gamma_1(x), \gamma_2(y)\}.
\end{aligned}
\]
By~\cite[Corollary 15.1.2]{Roc:70}, the polar of the gauge of $C$ is the support function of $C$, which is given by 
\[
\begin{aligned}
\sup_{x \in C_1, y \in C_2} \left\langle(x, y), (z_1,z_2) \right \rangle &= \sup_{x \in C_1} \left\langle x, z_1 \right\rangle +  \sup_{y \in C_2} \left\langle y, z_2 \right\rangle \\
&= \gamma_1^\circ(z_1) + \gamma_2^\circ(z_2).
\end{aligned}
\]
\end{proof}

This theorem allows us to easily compute the polars for $\phiSum$ and $\phiMax$
in terms of the polars of $\matrixNorm{\cdot}_*$ and $\|\cdot\|_1$,  which are
the dual norms, the spectral norm and infinity norm, respectively. 

\begin{corollary}[Explicit variational formulae for~\eqref{eq:sum-SPCP} and~\eqref{eq:max-SPCP}]
\label{cor:explicit-polar}
We have 
\begin{equation}
\label{eq:polarFormulas}
\begin{aligned}
\phiSum^\circ(Z_1, Z_2) &= \max\left\{\matrixNorm{Z_1}_2, \frac{1}{\lambdaSum} \|Z_2\|_\infty\right\}\\
\phiMax^\circ(Z_1, Z_2) &= \matrixNorm{Z_1}_2 + \frac{1}{\lambdaMax} \|Z_2\|_\infty,
\end{aligned}
\end{equation}
where $\matrixNorm{X}_2$ denotes the spectral norm (largest eigenvalue of $X^TX$). 

\end{corollary}

This result was also obtained by~\cite[Section 9]{BergFriedlander:2011}, but is stated only for norms. 
Theorem~\ref{thm:product-gauge} applies to gauges, and in particular now allows 
asymmetric gauges, so non-negativity constraints can be easily modeled.

We now have closed form solutions for $v'(\tau)$ in~\eqref{eq:vprime} for both
formulations of interest.  The remaining challenge is to design a fast solver
for~\eqref{eq:value} for formulations~\eqref{eq:sum-SPCP}
and~\eqref{eq:max-SPCP}.  We focus on this challenge in the remaining sections
of the paper.  We also discuss the advantage of~\eqref{eq:max-SPCP} from this
computational perspective.

\section{PROJECTIONS}
\label{sec:projections}
In this section, we consider the computational issues of projecting onto the
set defined by $\phi(L,S) \le \tau$. For $\phiMax(L, S) =
\max(\matrixNorm{L}_*, \lambdaMax \|S\|_1)$ this is straightforward since the
set is just the product set of the nuclear norm and $\ell_1$ norm balls, and
efficient projectors onto these are known. In particular, projecting an $m
\times n$ matrix (without loss of generality let $m \le n$) onto the nuclear
norm ball takes $\order( m^2n )$ operations, and projecting it onto the
$\ell_1$-ball can be done on $\order( mn )$ operations using fast
median-finding algorithms~\citep{knapsack1984,Duchi2008}.

For $\phiSum(L, S) = \matrixNorm{L}_* + \lambdaSum \|S\|_1$, the projection is
no longer straightforward. Nonetheless, the following lemma shows this
projection can be efficiently implemented.  
\begin{proposition}{\cite[Section 5.2]{BergFriedlander:2011}}
Projection onto
    the scaled $\ell_1$-ball, that is, $\{ x \in \R^d \mid \sum_{i=1}^d
\alpha_i |x_i| \le 1 \}$ for some $\alpha_i > 0$, can be done in $\order( d
\log(d) ) $ time.  
\end{proposition} 
The proof of the proposition follows by noting that
the solution can be written in a form depending only on a single scalar
parameter, and this scalar can be found by sorting $(|x_i|/\alpha_i)$ followed
by appropriate summations.
We conjecture that fast median-finding
ideas could reduce this to $\order(d)$ in theory, the same as the optimal
complexity for the $\ell_1$-ball. 

Armed with the above proposition, we state an important lemma below. For our
purposes, we may think of $S$ as a vector in $\R^{mn}$ rather than a matrix in
$\R^{m\times n}$.


\begin{lemma}{\cite[Section 9.2]{BergFriedlander:2011}}\label{lemma:jointProjection}
Let $L=U\Sigma V^T$ and $\Sigma = \diag(\sigma)$, and let $(S_i)_{i=1}^{mn}$ be
any ordering of the elements of $S$. Then the projection of $(L,S)$ onto the
$\phiSum$ ball is $(U\diag(\hat{\sigma})V^T,\hat{S})$, where
$(\hat{\sigma},\hat{S})$ is the projection onto the scaled $\ell_1$-ball $\{
    (\sigma,S) \mid\, \sum_{j=1}^{\min(m,n)} |\sigma_j| + \sum_{i=1}^{mn}
\lambdaSum |S_i| \le 1 \}$.
\end{lemma}
\begin{proof}[Sketch of proof]
We need to solve
\begin{equation*}
  \min_{\{ (L',S') \mid \, \phiSum(L',S') \le 1 \} }    
       \,\frac{1}{2}\norm{L'-L}_F^2 + \frac{1}{2}\norm{S'-S}_F^2.
\end{equation*}
Alternatively, solve
\[
  \min_{S'}\, \min_{\{L' \mid\, \matrixNorm{L'}_* \le 1 - \lambdaSum\|S'\|_1 \}}  \,\frac{1}{2}\norm{L'-L}_F^2 + \frac{1}{2}\norm{S'-S}_F^2.
\]
The inner minimization is equivalent to projecting onto the nuclear norm ball,
and this is well-known to be soft-thresholding of the singular values. Since it
depends only on the singular values, recombining the two minimization terms
gives exactly a joint projection onto a scaled $\ell_1$-ball.
\end{proof}

\begin{remark}
All the references to the $\ell_1$-ball can be replaced by the intersection of
the $\ell_1$-ball and the non-negative cone, and the projection is still
efficient.  As noted in Section~\ref{sec:variational}, imposing non-negativity
constraints is covered by the gauge results of Theorem~\ref{thm:product-gauge}
and Corollary~\ref{cor:explicit-polar}.  Therefore, both the variational and
efficient computational framework can be applied to this interesting case. 
\end{remark}

\section{SOLVING THE SUB-PROBLEM VIA PROJECTED QUASI-NEWTON METHODS}
\label{sec:QN}
In order to accelerate the approach, we can use quasi-Newton (QN) methods since
the objective has a simple structure.\footnote{ We use ``quasi-Newton'' to mean
an approximation to a Newton method and it should not be confused with methods
like BFGS} The main challenge here is that for the $\matrixNorm{L}_*$ term,
it is tricky to deal with a weighted quadratic term (whereas for $\|S\|_1$, we
can obtain a low-rank Hessian and solve it efficiently via coordinate descent).

We wish to solve \eqref{eq:max-SPCP-flip}.  Let $X=(L,S)$ be the full variable,
so we can write the objective function as $f(X)=\frac{1}{2}\|\AAA(X)-Y\|_F^2$.
To simplify the exposition, we take  $\AAA=(I, I)$ to be the $mn \times 2mn$
matrix, but the presented approach applies to general linear operators
(including terms like $\proj_\Omega$).  The matrix structure of $L$ and $S$ is
not important here, so we can think of them as $mn \times 1$ vectors instead of
$m \times n$ matrices.

The gradient is $\nabla f(X) = \AAA^T( \AAA(X)-Y )$. For convenience, we use
$r(X) = \AAA(X)-Y$ and
\[
\nabla f(X) = \pmat{ \nabla_L f(X) \\ \nabla_S f(X) } = 
\AAA^T\pmat{r(X)\\r(X)}, \quad r_k \equiv r(X_k). 
\]

The Hessian is  $\AAA^T\AAA = \pmat{I&I\\I&I}$.  We cannot simultaneously
project $(L,S)$ onto their constraints with this Hessian scaling (doing so
would solve the original problem!), since the Hessian removes separability.
Instead, we use $(L_k,S_k)$ to approximate the cross-terms.

\newcommand{\redd}[1]{\mathbf{#1}}
The true function is a quadratic, so the following quadratic expansion around $X_k=(L_k,S_k)$ is exact:
\[
\begin{aligned}
f(L,S) = f(X_k) &+ \iprod{\pmat{\nabla_L f(X_k)\\\nabla_S f(X_k)}}{\pmat{L-L_k\\S-S_k}} \\
&+ \iprod{\pmat{L-L_k\\S-S_k}}{\nabla^2 f\pmat{L-L_k\\S-S_k}} \\
    = f(X_k) &+ \iprod{\pmat{r_k\\r_k}}{\pmat{L-L_k\\S-S_k}} \\
    &+ \iprod{\pmat{L-L_k\\S-S_k}}{\pmat{1&1\\1&1}\pmat{L-L_k\\S-S_k}} \\
    = f(X_k)& + \iprod{\pmat{r_k\\r_k}}{\pmat{L-L_k\\S-S_k}} \\
    &+ \iprod{\pmat{\redd{L}-L_k\\\redd{S}-S_k}}{\pmat{L-L_k+\redd{S}-S_k\\\redd{L}-L_k+S-S_k}}
\end{aligned}
\]
The coupling of the second order terms, shown in bold, prevents direct 1-step
minimization of $f$, subject to the nuclear and 1-norm constraints.  The
FISTA~\citep{BecTeb09} and spectral gradient methods (SPG)~\citep{Wright2009}
replace the Hessian $\pmat{I&I\\I&I}$ with the upper bound $2\pmat{I&0\\0&I}$,
which solves the coupling issue, but potentially lose too much second order
information.  After comparing FISTA and SPG, we use the SPG method for solving
\eqref{eq:sum-SPCP-flip}.  However, for \eqref{eq:max-SPCP-flip} (and for
\eqref{eq:lag-SPCP}, which has no constraints but rather non-smooth terms,
which can be treated like constraints using proximity operators), the
constraints are uncoupled and we can take a ``middle road'' approach, replacing
\[
\iprod{\pmat{\redd{L}-L_k\\\redd{S}-S_k}}{\pmat{L-L_k+\redd{S}-S_k\\\redd{L}-L_k+S-S_k}}
\]
with
\[
\iprod{\pmat{L-L_k\\S-S_k}}{\pmat{L-L_k+\redd{S_k-S_{k-1}}\\\redd{L_{k+1}-L_{k}}+S-S_k}}.
\]
The first term is decoupled, allowing us to update $L_k$, and then this is
plugged into the second term in a Gauss-Seidel fashion. In practice, we also
scale this second-order term with a number slightly greater than $1$ but less
than $2$ (e.g.,~$1.25$) which leads to more robust behavior. We expect this
``quasi-Newton'' trick to do well when $S_{k+1}-S_k$ is similar to $S_k -
S_{k-1}$.

\section{NUMERICAL RESULTS}
\label{sec:numerics}
The numerical experiments are done with the algorithms suggested in this paper
as well as code from PSPG~\citep{AybatRPCA}, NSA~\citep{Aybat2013}, and
ASALM~\citep{ASALM}\footnote{PSPG, NSA and ASALM available from the experiment
    package at \url{http://www2.ie.psu.edu/aybat/codes.html}}.  
We modified the other software as needed for testing purposes. PSPG, NSA and
ASALM all solve \eqref{eq:sum-SPCP}, but ASALM has another variant which solves
\eqref{eq:lag-SPCP} so we test this as well. All three programs also use
versions of PROPACK from \cite{PROPACK,BeckerPROPACK} to compute partial SVDs.
Since the cost of a single iteration may vary among the solvers, we measure
error as a function of time, not iterations.  When a reference solution
$(L^\star,S^\star)$ is available, we measure the (relative) error of a trial
solution $(L,S)$ as $\|L-L^\star\|_F/\|L^\star\|_F +
\|S-S^\star\|_F/\|S^\star\|_F$.  The benchmark is designed so the time required
to calculate this error at each iteration does not factor into the reported
times.  Since picking stopping conditions is solver dependent, we show plots of
error vs time, rather than list tables.  All tests are done in Matlab and the
dominant computational time was due to matrix multiplications for all
algorithms; all code was run in the same quad-core $1.6$~GHz i7 computer.

For our implementations of the \eqref{eq:max-SPCP-flip},
\eqref{eq:sum-SPCP-flip} and \eqref{eq:lag-SPCP}, we use a randomized
SVD~\citep{halko2011finding}.  Since the number of singular values needed is not
known in advance, the partial SVD may be called several times (the same is true
for PSPG, NSA and ASALM).  Our code limits the number of singular values on the
first two iterations in order to speed up calculation without affecting
convergence.  Unfortunately, the delicate projection involved in
\eqref{eq:sum-SPCP-flip} makes incorporating a partial SVD to this setting more
challenging, so we use Matlab's dense \texttt{SVD} routine.

\subsection{Synthetic test with exponential noise}
We first provide a test with generated data. The observations $Y\in\R^{m\times
n}$ with $m=400$ and $n=500$ were created by first sampling a rank $20$ matrix
$Y_0$ with random singular vectors (i.e.,~from the Haar measure) and singular
values drawn from a uniform distribution with mean $0.1$, and then adding
exponential random noise (with mean equal to one tenth the median absolute
value of the entries of $Y_0$).  This exponential noise, which has a longer
tail than Gaussian noise, is expected to be captured partly by the $S$ term and
partly by the $\|L+S-Y\|_F$ term.

Given $Y$, the reference solution $(L^\star,S^\star)$ was generated by solving
\eqref{eq:lag-SPCP} to very high accuracy; the values $\lambdaL=0.25$ and
$\lambdaS=10^{-2}$ were picked by hand tuning $(\lambdaL,\lambdaS)$ to find a
value such that both $L^\star$ and $S^\star$ are non-zero.  The advantage to
solving \eqref{eq:lag-SPCP} is that knowledge of
$(L^\star,S^\star,\lambdaL,\lambdaS)$ allows us to generate the parameters for
all the other variants, and hence we can test different problem formulations.

With these parameters, $L^\star$ was rank 17 with nuclear norm $6.754$,
$S^\star$ had 54 non-zero entries (most of them positive) with $\ell_1$ norm
$0.045$, the normalized residual was $\|L^\star+S^\star-Y\|_F/\|Y\|_F=0.385$,
and $\eps=1.1086$, $\lambdaSum=0.04$, $\lambdaMax=150.0593$, $\tauSum=6.7558$
and $\tauMax=6.7540$.

\begin{figure}[ht]
\includegraphics[width=\columnwidth]{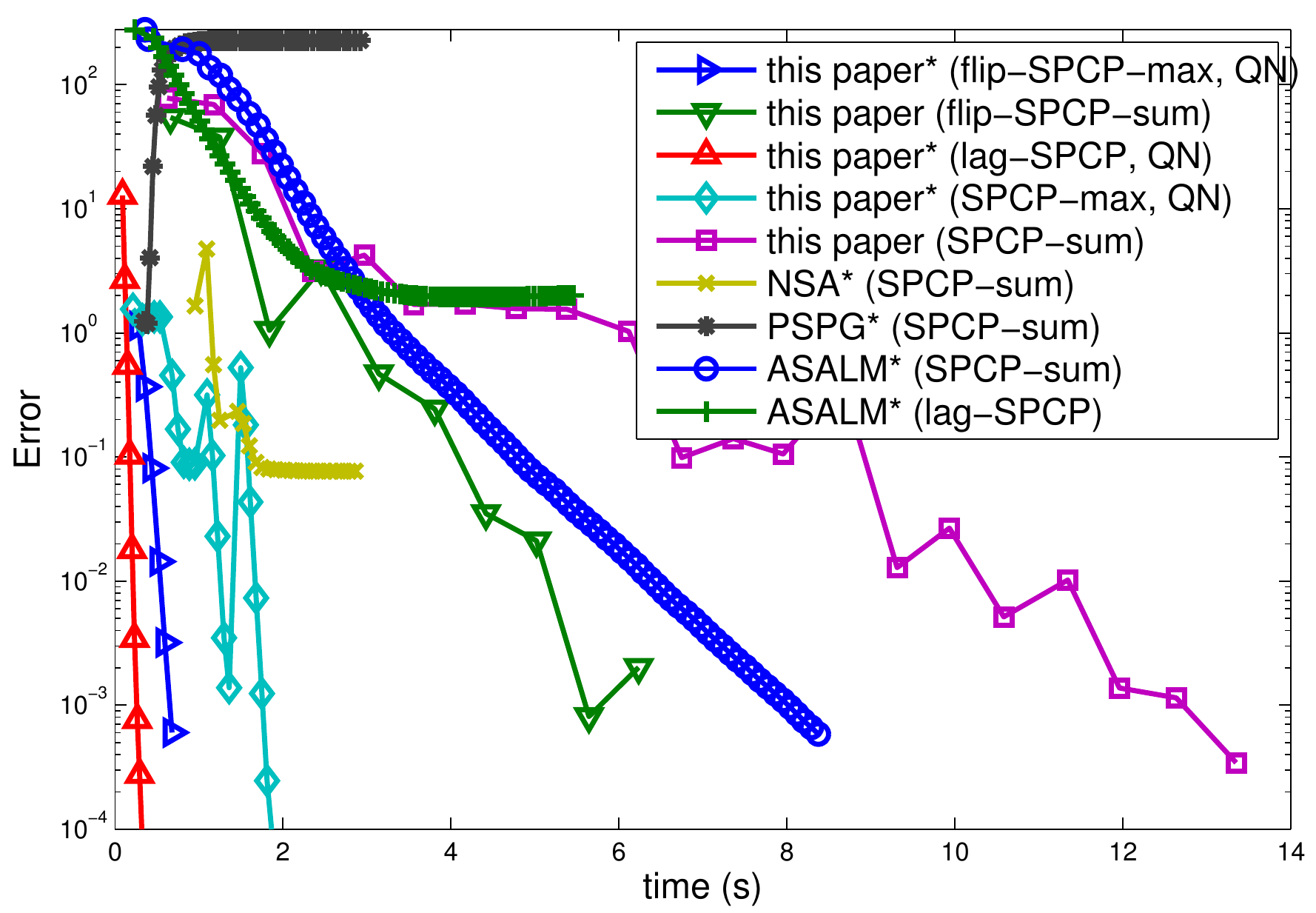}
\caption{The exponential noise test. The asterisk in the legend means the method uses a fast SVD.}
\label{fig:1}
\end{figure}

Results are shown in Fig.~\ref{fig:1}. Our methods for \eqref{eq:max-SPCP-flip}
and \eqref{eq:lag-SPCP} are extremely fast, because the simple nature of these
formulations allows the quasi-Newton acceleration scheme of
Section~\ref{sec:QN}.  In turn, since our method for
solving~\eqref{eq:max-SPCP} uses the variational framework of
Section~\ref{sec:variational} to solve a sequence of~\eqref{eq:max-SPCP-flip}
problems, it is also competitive (shown in cyan in Figure~\ref{fig:1}).  The
jumps are due to re-starting the sub-problem solver with a new value of $\tau$,
generated according to~\eqref{eq:newton}. 

Our proximal gradient method for \eqref{eq:sum-SPCP-flip}, which makes use of
the projection in Lemma~\ref{lemma:jointProjection}, converges more slowly,
since it is not easy to accelerate with the quasi-Newton scheme due to variable
coupling, and it does not make use of fast SVDs. Our solver
for~\eqref{eq:sum-SPCP}, which depends on a sequence of
problems~\eqref{eq:sum-SPCP-flip},  converges slowly. 

The ASALM performs reasonably well, which was unexpected since it was shown to
be worse than NSA and PSPG in \cite{AybatRPCA,Aybat2013}.  The PSPG solver
converges to the wrong answer, most likely due to a bad choice of the smoothing
parameter $\mu$; we tried choosing several different values other than the
default but did not see improvement for this test (for other tests, not shown,
tweaking $\mu$ helped significantly).  The NSA solver reaches moderate error
quickly but stalls before finding a highly accurate solution.

\subsection{Synthetic test from \cite{Aybat2013}}
We show some tests from the test setup of \cite{Aybat2013} in the $m=n=1500$
case. The default setting of $\lambdaSum=1/\sqrt{\max(m,n)}$ was used, and then
the NSA solver was run to high accuracy to obtain a reference solution
$(L^\star,S^\star)$.  From the knowledge of $(L^\star,S^\star,\lambdaSum)$, one
can generate $\lambdaMax,\tauSum,\tauMax,\eps$, but not $\lambdaS$ and
$\lambdaL$, and hence we did not test the solvers for \eqref{eq:lag-SPCP} in
this experiment.  The data was generated as $Y=L_0+S_0+Z_0$, where $L_0$ was
sampled by multiplication of $m \times r$ and $r \times n$ normal Gaussian
matrices, $S_0$ had $p$ randomly chosen entries uniformly distributed within
$[-100,100]$, and $Z_0$ was white noise chosen to give a SNR of $45$~dB.  We
show three tests that vary the rank from $\{0.05,0.1\}\cdot \min(m,n)$ and the
sparsity ranging from $p=\{0.05,0.1\}\cdot mn$.  Unlike \cite{Aybat2013}, who
report error in terms of a true noiseless signal $(L_0,S_0)$, we report the
optimization error relative to $(L^\star,S^\star)$.

For the first test (with $r=75$ and $p=0.05\times mn$), $L^\star$ had rank $786$
and nuclear norm $111363.9$; $S^\star$ had $75.49\%$ of its elements nonzero
and $\ell_1$ norm $ 5720399.4$, and
$\|L^\star+S^\star-Y^\star\|_F/\|Y\|_F=1.5\cdot 10^{-4}$.  The other parameters
were $\eps=3.5068$, $\lambdaSum=0.0258$, $\lambdaMax=0.0195$,
$\tauSum=2.5906\cdot 10^{5}$ and $\tauMax=1.1136\cdot 10^{5}$.  An interesting
feature of this test is that while $L_0$ is low-rank, $L^\star$ is nearly
low-rank but with a small tail of significant singular values until number
$786$.  We expect methods to converge quickly to low-accuracy where only a
low-rank approximation is needed, and then slow down as they try to find a
larger rank highly-accurate solution. 

\begin{figure}[ht]
\includegraphics[width=\columnwidth]{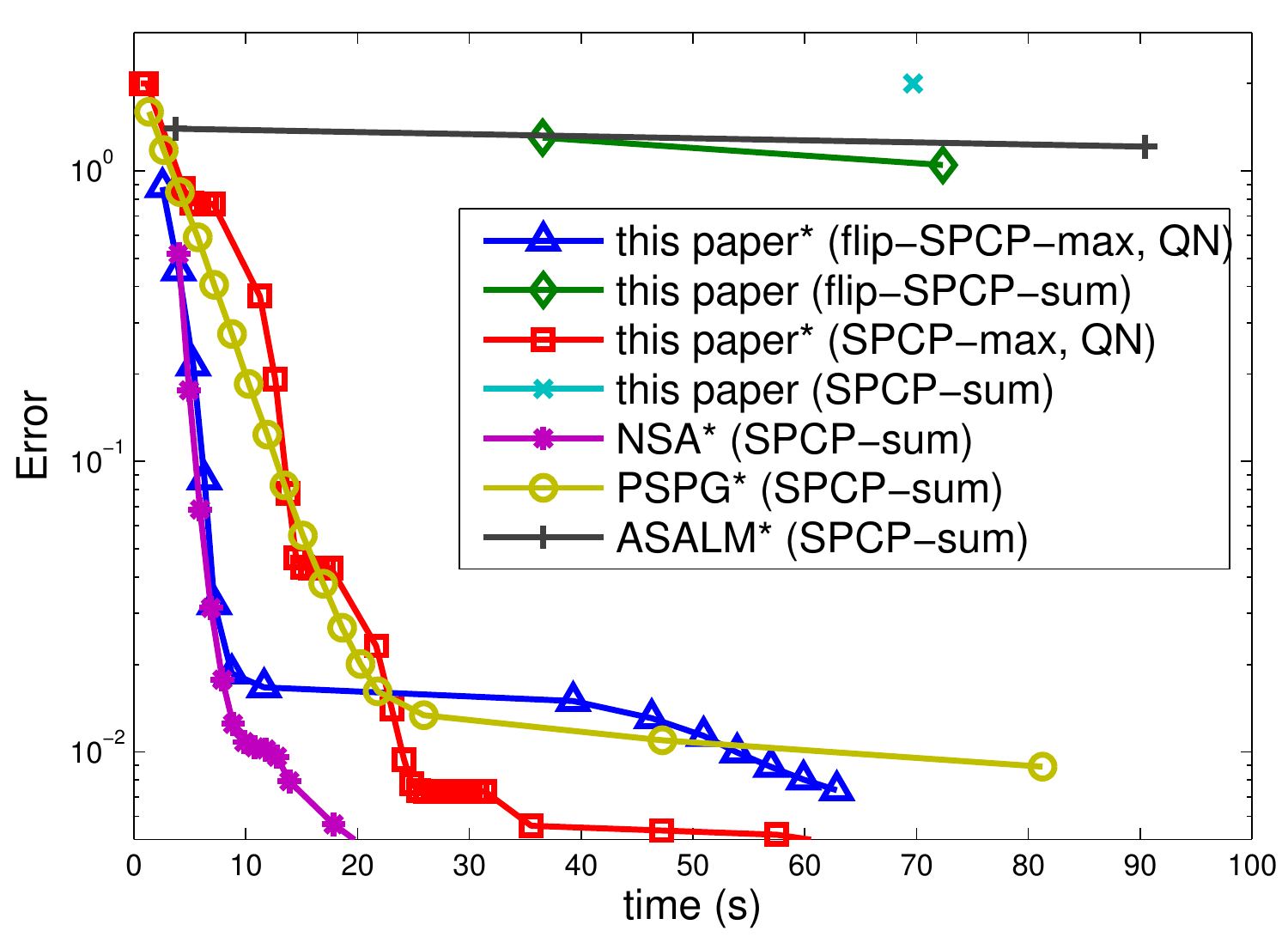}
\caption{The $1500\times 1500$ synthetic noise test.}
\label{fig:2}
\end{figure}

The results are shown in Fig.~\ref{fig:2}. Errors barely dip below $0.01$ (for
comparison, an error of $2$ is achieved by setting $L=S=0$).  The NSA and PSPG
solvers do quite well. In contrast to the previous test, ASALM does poorly.
Our methods for \eqref{eq:sum-SPCP-flip}, and hence \eqref{eq:sum-SPCP}, are
not competitive, since they use dense SVDs.  We imposed a time-limit of about
one minute, so these methods only manage a single iteration or two.  Our
quasi-Newton method for \eqref{eq:max-SPCP-flip} does well initially, then
takes a long time due to a long partial SVD computation.  Interestingly,
\eqref{eq:max-SPCP} does better than pure \eqref{eq:max-SPCP-flip}.  One
possible explanation is that it chooses a fortuitous sequence of $\tau$ values,
for which the corresponding \eqref{eq:max-SPCP-flip} subproblems become
increasingly hard, and therefore benefit from the warm-start of the solution of
the easier previous problem.  This is consistent with empirical observations
regarding continuation techniques, see
e.g.,~\citep{vandenberg2008probing,Wright2009}.

\begin{figure}[ht]
\includegraphics[width=\columnwidth]{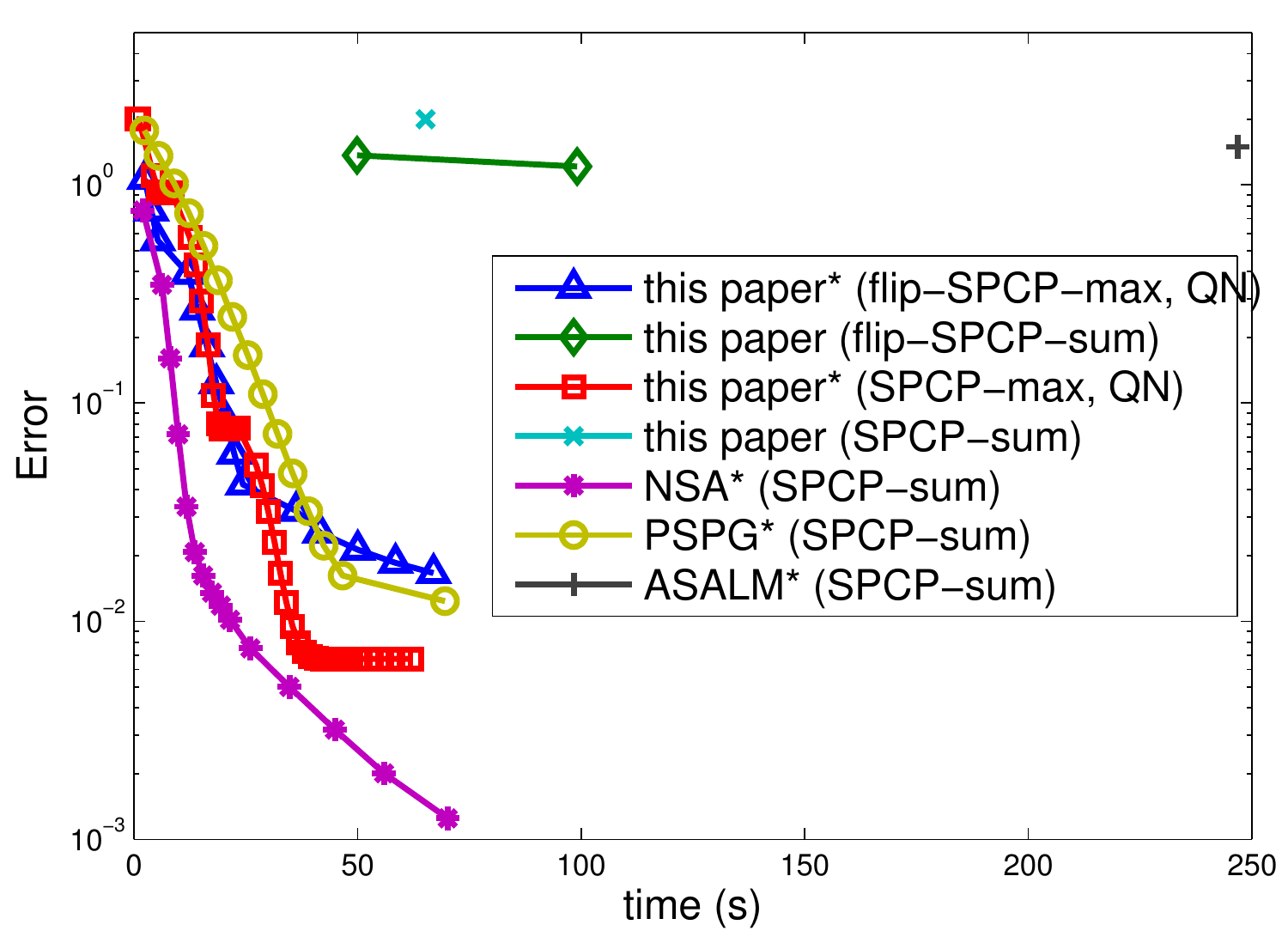}
\caption{Second $1500\times 1500$ synthetic noise test.}
\label{fig:3}
\end{figure}

Figure~\ref{fig:3} is the same test but with $r=150$ and $p=0.1\cdot m n$, and the conclusions are largely similar.

\subsection{Cloud removal}
\begin{figure*}[ht]
\centering
\includegraphics[width=1.8\columnwidth]{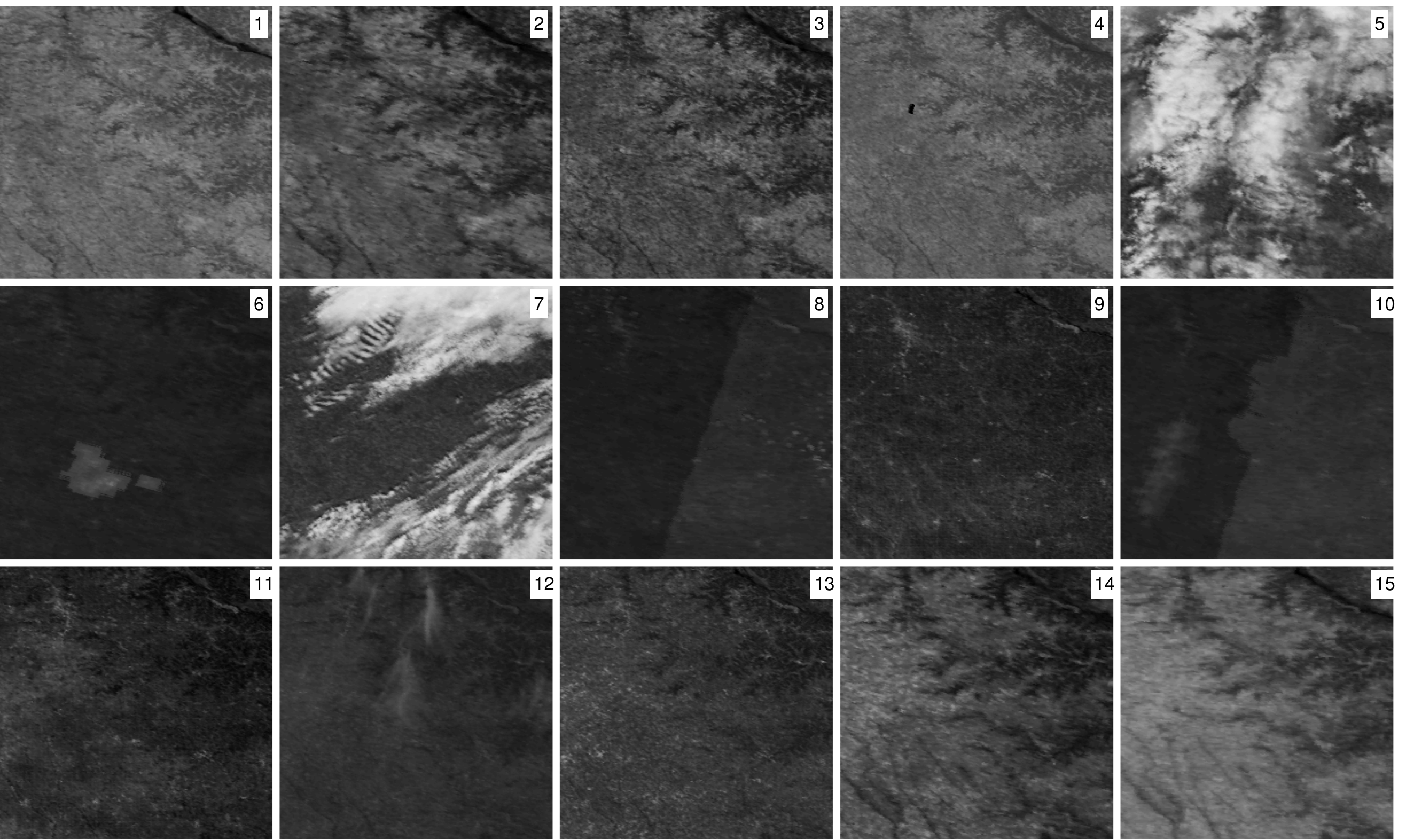}
\caption{Satellite photos of the same location on different days}
\label{fig:4}
\end{figure*}

\begin{figure}[t]
\includegraphics[width=\columnwidth]{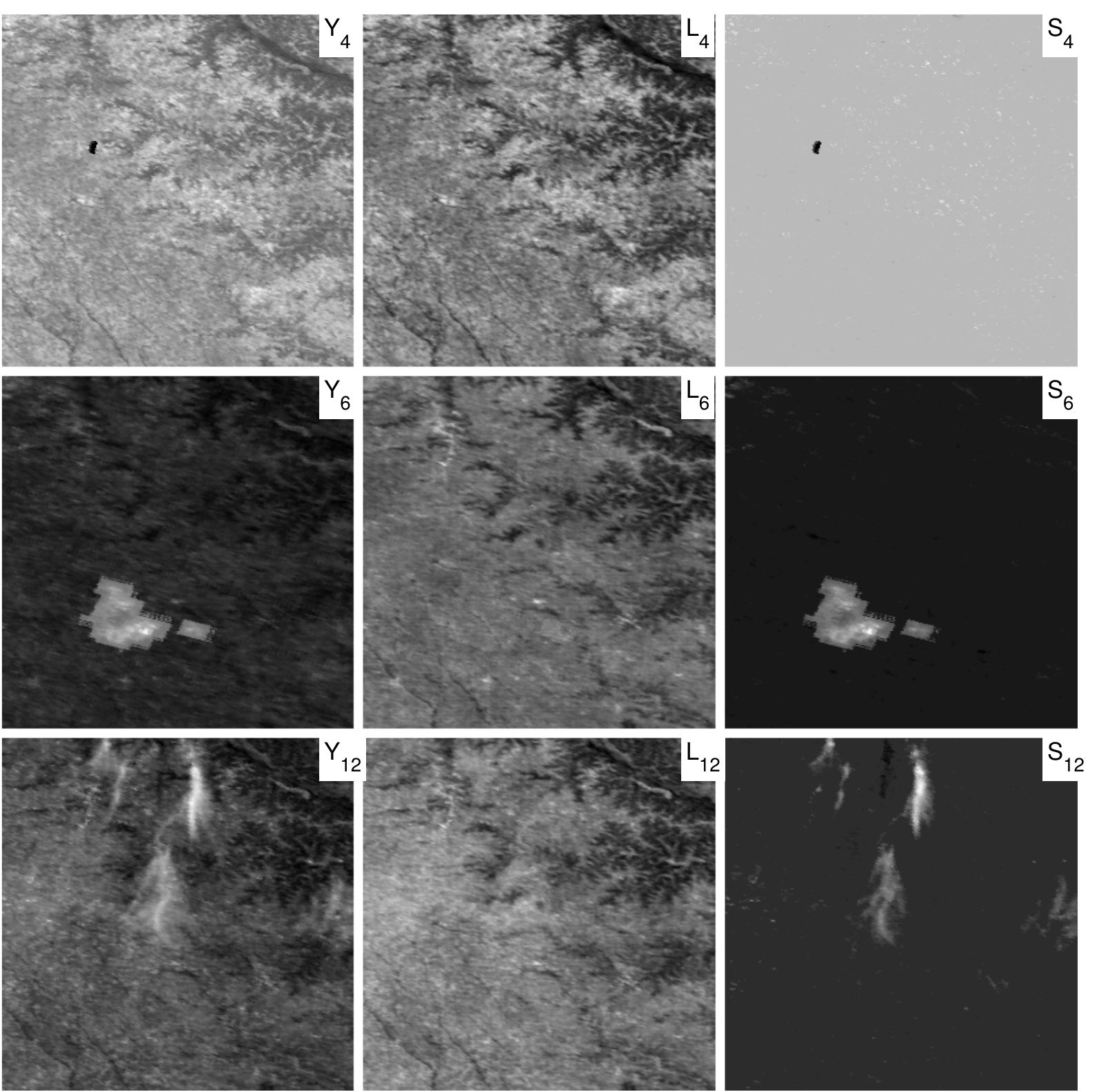}
\caption{Showing frames 4, 5 and 12. Leftmost column is original data, middle column is low-rank term of the solution, and right column is sparse term of the solution. Data have been processed slightly to enhance contrast for viewing.}
\label{fig:5}
\end{figure}

Figure~\ref{fig:4} shows 15 images of size $300 \times 300$ from the MODIS satellite,\footnote{Publicly available at \url{http://ladsweb.nascom.nasa.gov/}} 
after some transformations to turn images from different spectral bands into
one grayscale images.  Each image is a photo of the same rural location but at
different points in time over the course of a few months.  The background
changes slowly and the variability is due to changes in vegetation, snow cover,
and different reflectance.  There are also outlying sources of error, mainly
due to clouds (e.g., major clouds in frames 5 and 7, smaller clouds in frames
9, 11 and 12), as well as artifacts of the CCD camera on the satellite (frame 4
and 6) and issues stitching together photos of the same scene (the lines in
frames 8 and 10).

There are hundreds of applications for clean satellite imagery, so removing the
outlying error is of great practical importance.  Because of slow changing
background and sparse errors, we can model the problem using the robust PCA
approach.  We use the \eqref{eq:max-SPCP-flip} version due to its speed, and
pick parameters $(\lambdaMax,\tauMax)$ by using a Nelder-Mead simplex search.
For an error metric to use in the parameter tuning, we remove frame $1$ from
the data set (call it $y_1$) and set $Y$ to be frames 2--15.  From this
training data $Y$, the algorithm generates $L$ and $S$. Since $L$ is a $300^2
\times 14$ matrix, it has far from full column span.  Thus our error is the
distance of $y_1$ from the span of $L$, i.e., $\|y_1 -
\proj_{\text{span}(L)}(y_1)\|_2$.

Our method takes about 11 iterations and 5 seconds, and uses a dense SVD
instead of the randomized method due to the high aspect ratio of the
matrix.  Some results of the obtained $(L,S)$ outputs are in Fig.~\ref{fig:5},
where one can see that some of the anomalies in the original data frames $Y$
are picked up by the $S$ term and removed from the $L$ term.  Frame 4 has what
appears to be a camera pixel error;  frame 6 has another artificial error (that
is, caused by the camera and not the scene); and frame 12 has cloud cover.

\section{CONCLUSIONS}

In this paper, we reviewed several formulations and algorithms for 
the RPCA problem. We introduced a new denoising formulation~\eqref{eq:max-SPCP}
to the ones previously considered, and discussed modeling and algorithmic advantages
of denoising formulations~\eqref{eq:max-SPCP} and~\eqref{eq:sum-SPCP} 
compared to flipped versions~\eqref{eq:max-SPCP-flip} and~\eqref{eq:sum-SPCP-flip}.
In particular, we showed that these formulations can be linked using a variational framework,
which can be exploited to solve denoising formulations using a sequence of flipped problems. 
For~\eqref{eq:max-SPCP-flip}, we proposed a quasi-Newton acceleration that is competitive 
with state of the art, and used this innovation to design a fast method for~\eqref{eq:max-SPCP}
through the variational framework. The new methods were compared against prior art 
on synthetic examples, and applied to a real world cloud removal application application using publicly 
available MODIS satellite data. 

\label{sec:conclusions}

\small{
\pdfbookmark[1]{References}{refSection}
\bibliographystyle{icml2013}
\bibliography{references_RPCA}

\begin{thebibliography}{32}
\providecommand{\natexlab}[1]{#1}
\providecommand{\url}[1]{\texttt{#1}}
\expandafter\ifx\csname urlstyle\endcsname\relax
  \providecommand{\doi}[1]{doi: #1}\else
  \providecommand{\doi}{doi: \begingroup \urlstyle{rm}\Url}\fi

\bibitem[Aravkin et~al.(2013{\natexlab{a}})Aravkin, Burke, and
  Friedlander]{AravkinBurkeFriedlander:2013}
Aravkin, A.~Y., Burke, J., and Friedlander, M.~P.
\newblock Variational properties of value functions.
\newblock \emph{SIAM J. Optimization}, 23\penalty0 (3):\penalty0 1689--1717,
  2013{\natexlab{a}}.

\bibitem[Aravkin et~al.(2013{\natexlab{b}})Aravkin, Kumar, Mansour, Recht, and
  Herrmann]{Aravkin2013}
Aravkin, A.~Y., Kumar, R., Mansour, H., Recht, B., and Herrmann, F.~J.
\newblock {A robust SVD-free approach to matrix completion, with applications
  to interpolation of large scale data}.
\newblock 2013{\natexlab{b}}.
\newblock URL \url{http://arxiv.org/abs/1302.4886}.

\bibitem[Aybat et~al.(2013)Aybat, Goldfarb, and Ma]{AybatRPCA}
Aybat, N., Goldfarb, D., and Ma, S.
\newblock Efficient algorithms for robust and stable principal component
  pursuit.
\newblock \emph{Computational Optimization and Applications, {\em (accepted)}},
  2013.

\bibitem[Aybat \& Iyengar(2014)Aybat and Iyengar]{Aybat2013}
Aybat, N.~S. and Iyengar, G.
\newblock An alternating direction method with increasing penalty for stable
  principal component pursuit.
\newblock \emph{Computational Optimization and Applications, {\em
  (submitted)}}, 2014.
\newblock \url{http://arxiv.org/abs/1309.6553}.

\bibitem[Baldassarre et~al.(2013)Baldassarre, Cevher, McCoy, Tran~Dinh, and
  Asaei]{McCoySPMAG}
Baldassarre, L., Cevher, V., McCoy, M., Tran~Dinh, Q., and Asaei, A.
\newblock Convexity in source separation: Models, geometry, and algorithms.
\newblock Technical report, 2013.
\newblock \url{http://arxiv.org/abs/1311.0258}.

\bibitem[Beck \& Teboulle(2009)Beck and Teboulle]{BecTeb09}
Beck, A. and Teboulle, M.
\newblock {A Fast Iterative Shrinkage-Thresholding Algorithm for Linear Inverse
  Problems}.
\newblock \emph{SIAM J. Imaging Sciences}, 2\penalty0 (1):\penalty0 183--202,
  January 2009.

\bibitem[Becker \& Cand\`es(2008)Becker and Cand\`es]{BeckerPROPACK}
Becker, S. and Cand\`es, E.
\newblock Singular value thresholding toolbox, 2008.
\newblock Available from \url{http://svt.stanford.edu/}.

\bibitem[Brucker(1984)]{knapsack1984}
Brucker, P.
\newblock An {O}(n) algorithm for quadratic knapsack problems.
\newblock \emph{Operations Res. Lett.}, 3\penalty0 (3):\penalty0 163 -- 166,
  1984.
\newblock \doi{10.1016/0167-6377(84)90010-5}.

\bibitem[Cand\`{e}s et~al.(2011)Cand\`{e}s, Li, Ma, and Wright]{CanLiMa:11}
Cand\`{e}s, E.~J., Li, X., Ma, Y., and Wright, J.
\newblock Robust principal component analysis?
\newblock \emph{J. Assoc. Comput. Mach.}, 58\penalty0 (3):\penalty0 1--37, May
  2011.

\bibitem[Chandrasekaran et~al.(2009)Chandrasekaran, Sanghavi, Parrilo, and
  Willsky]{ChaSanPar:09}
Chandrasekaran, V., Sanghavi, S., Parrilo, P.~A., and Willsky, A.~S.
\newblock Sparse and low-rank matrix decompositions.
\newblock In \emph{SYSID 2009}, Saint-Malo, France, July 2009.

\bibitem[Chandrasekaran et~al.(2012)Chandrasekaran, Parrilo, and
  Willsky]{CPW12}
Chandrasekaran, V., Parrilo, P.~A., and Willsky, A.~S.
\newblock Latent variable graphical model selection via convex optimization.
\newblock \emph{Ann. Stat.}, 40\penalty0 (4):\penalty0 1935--2357, 2012.

\bibitem[Chen et~al.(2013)Chen, He, Ye, and Yuan]{chen2013direct}
Chen, Caihua, He, Bingsheng, Ye, Yinyu, and Yuan, Xiaoming.
\newblock The direct extension of admm for multi-block convex minimization
  problems is not necessarily convergent.
\newblock \emph{Optimization Online}, 2013.

\bibitem[Combettes \& Pesquet(2007)Combettes and Pesquet]{Combettes2007}
Combettes, P.~L. and Pesquet, J.-C.
\newblock {A Douglas–Rachford Splitting Approach to Nonsmooth Convex
  Variational Signal Recovery}.
\newblock \emph{IEEE J. Sel. Topics Sig. Processing}, 1\penalty0 (4):\penalty0
  564--574, December 2007.

\bibitem[Duchi et~al.(2008)Duchi, Shalev-Shwartz, Singer, and
  Chandra]{Duchi2008}
Duchi, J., Shalev-Shwartz, S., Singer, Y., and Chandra, T.
\newblock {Efficient projections onto the l1-ball for learning in high
  dimensions}.
\newblock In \emph{Intl. Conf. Machine Learning (ICML)}, pp.\  272--279, New
  York, July 2008. ACM Press.

\bibitem[Ganesh et~al.(2009)Ganesh, Lin, Wright, Wu, Chen, and Ma]{GaneshRPCA}
Ganesh, A., Lin, Z., Wright, J., Wu, L., Chen, M., and Ma, Y.
\newblock Fast algorithms for recovering a corrupted low-rank matrix.
\newblock In \emph{Computational Advances in Multi-Sensor Adaptive Processing
  (CAMSAP)}, pp.\  213--215, Aruba, Dec. 2009.

\bibitem[Halko et~al.(2011)Halko, Martinsson, and Tropp]{halko2011finding}
Halko, N., Martinsson, P.-G., and Tropp, J.~A.
\newblock Finding structure with randomness: Probabilistic algorithms for
  constructing approximate matrix decompositions.
\newblock \emph{SIAM review}, 53\penalty0 (2):\penalty0 217--288, 2011.

\bibitem[Huber(2004)]{Hub}
Huber, P.~J.
\newblock \emph{Robust Statistics}.
\newblock John Wiley and Sons, 2 edition, 2004.

\bibitem[Kyrillidis \& Cevher(2014)Kyrillidis and Cevher]{kyrillidis2014matrix}
Kyrillidis, Anastasios and Cevher, Volkan.
\newblock Matrix recipes for hard thresholding methods.
\newblock \emph{Journal of Mathematical Imaging and Vision}, 48\penalty0
  (2):\penalty0 235--265, 2014.

\bibitem[Larsen(1998)]{PROPACK}
Larsen, R.~M.
\newblock Lanczos bidiagonalization with partial reorthogonalization.
\newblock Tech. Report. DAIMI PB-357, Department of Computer Science, Aarhus
  University, September 1998.

\bibitem[Lee et~al.(2010)Lee, Recht, Salakhutdinov, Srebro, and
  Tropp]{JasonLee}
Lee, J., Recht, B., Salakhutdinov, R., Srebro, N., and Tropp, J.A.
\newblock Practical large-scale optimization for max-norm regularization.
\newblock In \emph{Neural Information Processing Systems (NIPS)}, Vancouver,
  2010.

\bibitem[Lin et~al.(2010)Lin, Chen, and Ma]{lin2010augmented}
Lin, Z., Chen, M., and Ma, Y.
\newblock The augmented {L}agrange multiplier method for exact recovery of
  corrupted low-rank matrices.
\newblock \emph{arXiv preprint arXiv:1009.5055}, 2010.

\bibitem[Mansour \& Vetro(2014)Mansour and Vetro]{Mansour2014}
Mansour, H. and Vetro, A.
\newblock Video background subtraction using semi-supervised robust matrix
  completion.
\newblock In \emph{To appear in IEEE International Conference on Acoustics,
  Speech, and Signal Processing (ICASSP)}, 2014.

\bibitem[Peng et~al.(2012)Peng, Ganesh, Wright, Xu, and Ma]{PenGanWri:12}
Peng, Y., Ganesh, A., Wright, J., Xu, W., and Ma, Y.
\newblock {RASL}: Robust alignment by sparse and low-rank decomposition for
  linearly correlated images.
\newblock \emph{IEEE Trans. Pattern Analysis and Machine Intelligence},
  34\penalty0 (11):\penalty0 2233--2246, 2012.

\bibitem[Recht \& R{\'e}(2011)Recht and R{\'e}]{recht2011parallel}
Recht, Benjamin and R{\'e}, Christopher.
\newblock Parallel stochastic gradient algorithms for large-scale matrix
  completion.
\newblock \emph{Math. Prog. Comput.}, pp.\  1--26, 2011.

\bibitem[Rockafellar(1970)]{Roc:70}
Rockafellar, R.~T.
\newblock \emph{Convex analysis}.
\newblock Princeton Mathematical Series, No. 28. Princeton University Press,
  Princeton, N.J., 1970.

\bibitem[Shen et~al.(2014)Shen, Wen, and Zhang]{Shen2014}
Shen, Y., Wen, Z., and Zhang, Y.
\newblock {Augmented Lagrangian alternating direction method for matrix
  separation based on low-rank factorization}.
\newblock \emph{Optimization Methods and Software}, 29\penalty0 (2):\penalty0
  239--263, March 2014.

\bibitem[Tao \& Yuan(2011)Tao and Yuan]{ASALM}
Tao, M. and Yuan, X.
\newblock Recovering low-rank and sparse components of matrices from incomplete
  and noisy observations.
\newblock \emph{SIAM J. Optimization}, 21:\penalty0 57--81, 2011.

\bibitem[van~den Berg \& Friedlander(2008)van~den Berg and
  Friedlander]{vandenberg2008probing}
van~den Berg, E. and Friedlander, M.~P.
\newblock Probing the {P}areto frontier for basis pursuit solutions.
\newblock \emph{SIAM J. Sci. Computing}, 31\penalty0 (2):\penalty0 890--912,
  2008.
\newblock software: \url{http://www.cs.ubc.ca/~mpf/spgl1/}.

\bibitem[van~den Berg \& Friedlander(2011)van~den Berg and
  Friedlander]{BergFriedlander:2011}
van~den Berg, E. and Friedlander, M.~P.
\newblock Sparse optimization with least-squares constraints.
\newblock \emph{SIAM J. Optimization}, 21\penalty0 (4):\penalty0 1201--1229,
  2011.

\bibitem[Wright et~al.(2009{\natexlab{a}})Wright, Ganesh, Rao, and
  Ma]{RPCA_algo_Wright}
Wright, J., Ganesh, A., Rao, S., and Ma, Y.
\newblock Robust principal component analysis: Exact recovery of corrupted
  low-rank matrices by convex optimization.
\newblock In \emph{Neural Information Processing Systems (NIPS)},
  2009{\natexlab{a}}.

\bibitem[Wright et~al.(2009{\natexlab{b}})Wright, Nowak, and
  Figueiredo]{Wright2009}
Wright, S.~J., Nowak, R.~D., and Figueiredo, M. A.~T.
\newblock {Sparse Reconstruction by Separable Approximation}.
\newblock \emph{IEEE Trans. Sig. Processing}, 57\penalty0 (7):\penalty0
  2479--2493, July 2009{\natexlab{b}}.

\bibitem[Zhang et~al.(2011)Zhang, Liang, Ganesh, and Ma]{ZhaLiaGan:11}
Zhang, Z., Liang, X., Ganesh, A., and Ma, Y.
\newblock {TILT}: Transform invariant low-rank textures.
\newblock In Kimmel, R., Klette, R., and Sugimoto, A. (eds.), \emph{Computer
  Vision -- ACCV 2010}, volume 6494 of \emph{Lecture Notes in Computer
  Science}, pp.\  314--328. Springer, 2011.

\end{thebibliography}
}

\end{document}